\documentclass[12pt]{amsart}
\usepackage{amssymb}
\usepackage{enumerate}
\usepackage{url}

\allowdisplaybreaks

\title{Generic forms of low Chow rank}
\author{Douglas A. Torrance}
\email{dtorrance@piedmont.edu}
\address{Department of Mathematics and Physics \\ Piedmont College \\ PO Box 10 \\ 1021 Central Ave \\ Demorest, GA 30535}

\newtheorem{theorem}{Theorem}[section]
\newtheorem{conjecture}[theorem]{Conjecture}
\newtheorem{lemma}[theorem]{Lemma}
\newtheorem{corollary}[theorem]{Corollary}
\newtheorem{proposition}[theorem]{Proposition}
\theoremstyle{definition}
\newtheorem{definition}[theorem]{Definition}

    \newtheoremstyle{TheoremNum}
        {\topsep}{\topsep}              
        {\itshape}                      
        {}                              
        {\bfseries}                     
        {.}                             
        {5pt plus 1pt minus 1pt}        
        {\thmname{#1}\thmnote{ \bfseries #3}}
    \theoremstyle{TheoremNum}
    \newtheorem{thmn}{Theorem}

\newcommand{\kk}{\Bbbk}
\renewcommand{\AA}{\mathbb A}
\newcommand{\NN}{\mathbb N}
\newcommand{\PP}{\mathbb P}

\newcommand{\define}[1]{\textit{#1}}
\newcommand{\affinecone}[1]{\widehat{#1}}
\newcommand{\tuple}[1]{\textbf{#1}}

\newcommand{\statement}[1]{\mathfrak{#1}}
\renewcommand{\vec}{\textbf}

\DeclareMathOperator{\chowrank}{rk_{Ch}}
\DeclareMathOperator{\waringrank}{rk}
\DeclareMathOperator{\Ch}{Ch}
\DeclareMathOperator{\Split}{Split}
\DeclareMathOperator{\expdim}{expdim}
\DeclareMathOperator{\Span}{span}

\begin{document}

\maketitle

\begin{abstract}
The least number of products of linear forms that may be added together to obtain a given form is the Chow rank of this form.  The Chow rank of a generic form corresponds to the smallest $s$ for which the $s$th secant variety of the Chow variety fills the ambient space.  We show that, except for certain known exceptions, this secant variety has the expected dimension for low values of $s$.
\end{abstract}

\section{Introduction}
Consider a homogeneous polynomial $f$ of degree $d$ in $n+1$ variables with coefficients in an algebraically closed field of characteristic zero.  As in \cite{Landsberg2}, we define the \define{Chow rank} of $f$, or $\chowrank f$, to be the minimum $s$ such that
\begin{equation*}
f = \ell_{1,1}\cdots\ell_{1,d}+\cdots+\ell_{s,1}\cdots\ell_{s,d}
\end{equation*}
where the $\ell_{i,j}$ are linear forms.

The Chow rank of a form is useful for determining the computational complexity of its evaluation.  Indeed, if we have the above decomposition, we can evaluate each $\ell_{i,j}$ using at most $sd(n+1)$ additions, and then evaluate $f$ by $sd$ multiplications followed by $s$ additions.  We say that $f$ is computable by a \textit{homogeneous circuit} of size $s+sd(n+2)$.  See \cite[\S 8]{Landsberg2} for more on this topic.

Certainly, the Chow rank of a monomial is 1, and as $f$ is the sum of at most $\binom{n+d}{d}$ monomials, we have $\chowrank f\leq\binom{n+d}{d}$.

Furthermore, recall that the \define{Waring rank} of $f$, or $\waringrank f$, is the minimum $s$ such that
\begin{equation*}
f=\ell_1^d+\cdots+\ell_s^d
\end{equation*}
where the $\ell_i$ are linear forms.  We see then that $\chowrank f\leq\waringrank f$.

The Waring rank of a form has been well-studied.  If $f$ is generic, then $\waringrank f$ is equal to the smallest $s$ such that $\sigma_s(\nu_d(\PP^n))$, the $s$th secant variety of the Veronese variety, fills the ambient space.  In \cite{AlexanderHirschowitz}, Alexander and Hirschowitz proved a century-old conjecture that these secant varieties have the expected dimension except for a small number of exceptional cases.

Similarly, the Chow rank of a generic form is equal to the smallest $s$ such that $\sigma_s(\Split_d(\PP^n))$, the $s$th secant variety of the Chow variety (the variety of all forms which can be completely reduced into a product of linear forms), fills the ambient space.  We define the \define{expected dimension} of $\sigma_s(\Split_d(\PP^n))$ to be
\begin{equation*}
\expdim\sigma_s(\Split_d(\PP^n))=\min\left\{s(dn+1),\binom{n+d}{d}\right\}-1.
\end{equation*}

Based on a na\"ive dimension count, we would expect that
\begin{equation*}
\dim\sigma_s(\Split_d(\PP^n))=\expdim\sigma_s(\Split_d(\PP^n)).
\end{equation*}

If this equation holds, then $\sigma_s(\Split_d(\PP^n))$ is \define{nondefective}.  Otherwise, it is \define{defective}.

Secant varieties of Chow varieties are related to unions of linear star configurations \cite{Shin2, Shin} and complete intersections on hypersurfaces \cite{carlini2008complete,carlini2010complete}.  There has also been some work on the more general problem of secant varieties to varieties of forms which can be reduced into a product of lower degree (not necessarily linear) forms \cite{catalisano2015secant,catalisano2014secant}.

The following conjecture is made in \cite{ArrondoBernardi}.

\begin{conjecture}\label{conjecture}
The secant variety $\sigma_s(\Split_d(\PP^n))$ is nondefective unless $d=2$ and $2\leq s\leq\frac{n}{2}$.
\end{conjecture}

Note that if this is true, then a generic form of degree $d\neq 2$ in $n+1$ variables has Chow rank $\left\lceil\frac{1}{dn+1}\binom{n+d}{d}\right\rceil$.

We summarize the previous progress towards a proof of Conjecture \ref{conjecture}.  It is trivial for linear and binary forms and straightforward for quadratics (see Theorem \ref{DEqualsTwo} below).  In \cite{ArrondoBernardi}, Arrondo and Bernardi showed that is is true if $n\geq 3(s-1)$.  In \cite{catalisano2015secant}, Catalisano et al. improved this to $n\geq 2s-1$ and also proved the conjecture for $s\geq\binom{n+d-1}{n}$.  Shin \cite{Shin} and Abo \cite{Abo} proved it for ternary forms.  Abo also provided partial results for cubics and quaternary forms.  As we will use these results in section 4, we state them here.

\begin{theorem}\label{AboResult}\cite{Abo}
Consider the following functions.
\begin{align*}
s_1(d) &=
\begin{cases}
\frac{1}{18}d^2+\frac{1}{6}d+1&\text{ if }d\equiv 0\pmod{6}\\
\frac{1}{18}d^2+\frac{2}{9}d-\frac{5}{18}&\text{ if }d\equiv 1\pmod{6}\\
\frac{1}{18}d^2+\frac{5}{18}d+\frac{2}{9}&\text{ if }d\equiv 2,5\pmod{6}\\
\frac{1}{18}d^2+\frac{1}{6}d&\text{ if }d\equiv 3\pmod{6}\\
\frac{1}{18}d^2+\frac{2}{9}d+\frac{2}{9}&\text{ if }d\equiv 4\pmod{6}\\
\end{cases}\\
s_2(d) &=
\begin{cases}
\frac{1}{18}d^2+\frac{1}{3}d+1&\text{ if }d\equiv 0\pmod{6}\\
\frac{1}{18}d^2+\frac{7}{18}d+\frac{14}{9}&\text{ if }d\equiv 1\pmod{6}\\
\frac{1}{18}d^2+\frac{4}{9}d+\frac{8}{9}&\text{ if }d\equiv 2\pmod{6}\\
\frac{1}{18}d^2+\frac{1}{3}d+\frac{1}{2}&\text{ if }d\equiv 3\pmod{6}\\
\frac{1}{18}d^2+\frac{7}{18}d+\frac{5}{9}&\text{ if }d\equiv 4\pmod{6}\\
\frac{1}{18}d^2+\frac{4}{9}d+\frac{7}{18}&\text{ if }d\equiv 5\pmod{6}\\
\end{cases}
\end{align*}
If $n=3$ and $s\leq s_1(d)$ or $s\geq s_2(d)$, or $d=3$ and $s\leq s_1(n)$ or $s\geq s_2(n)$, then $\sigma_s(\Split_d(\PP^n))$ is nondefective.
\end{theorem}

In section 2 of this paper, we introduce the basic definitions.  In section 3, we establish a method of induction which will be used in section 4 to prove the following results.

\begin{theorem}\label{induction on n}
If, for some $n_0\in\NN$, $s(dn_0+1)\leq\binom{n_0+d}{d}$ and $\sigma_s(\Split_d(\PP^{n_0}))$ is nondefective, then $\sigma_s(\Split_d(\PP^n))$ is nondefective for all $n\geq n_0$.
\end{theorem}

\begin{theorem}\label{SUpperBound}
If $s\leq 35$, then $\sigma_s(\Split_d(\PP^n))$ is nondefective for all $n,d\in\NN$ unless $d=2$ and $2\leq s\leq\frac{n}{2}$.
\end{theorem}

\begin{corollary}
If $f$ is a generic form of degree $d$ in $n+1$ variables and
\begin{itemize}
\item $n=3$ and $d\leq 22$ or $d=3$ and $n\leq 22$,
\item $n=4$ and $d\leq 11$ or $d=4$ and $n\leq 11$,
\item $n=5$ and $d\leq 8$ or $d=5$ and $n\leq 8$, or
\item $n=d=6$,
\end{itemize}
then $\chowrank f=\left\lceil\frac{1}{dn+1}\binom{n+d}{d}\right\rceil$.
\end{corollary}

This paper is based on results from the author's Ph.D. thesis \cite{Torrance} while studying at the University of Idaho.  Many thanks go to his advisor, Hirotachi Abo, and committee members Jennifer Johnson-Leung and Alexander Woo.  Thanks also go to Enrico Carlini, Anthony Geramita, Zach Teitler, William J. Turner, Jia Wan, and the anonymous referee for helpful comments.  Finally, thanks go to Logan Mayfield and the Monmouth College Department of Mathematics and Computer Science for the use of the server on which the calculations were run.

\section{Preliminaries}

Let $\kk$ be an algebraically closed field of characteristic zero and $R$ the polynomial ring $\kk[x_0,\ldots,x_n]$ with the usual grading.  Let $R_d$ denote the $d$th graded piece of $R$, i.e., the $\kk$-vector space of all forms in $R$ of degree $d$.  For a fixed $d$, define $N=\binom{n+d}{d}-1$.  Then $\dim_{\kk} R_d=N+1$ and so $\PP R_d=\PP^N$.  For all $f\in R_d$, denote by $[f]$ the equivalence class in $\PP^N$ containing $f$.

We define the \define{linear span} of $s$ varieties $X_1,\ldots,X_s\subset\PP^N$ to be the smallest linear variety containing $X_1,\ldots,X_s$ and denote it as $\langle X_1,\ldots,X_s\rangle$.

Also, for any variety $X\subset\PP^N$, we denote its affine cone in $\AA^{n+1}$ as $\affinecone X$.  Note that if $X$ is a linear subspace of $\PP^N$, then $\affinecone X$ is a $\kk$-vector space.  We denote the tangent space of $X$ at a point $x$ by $T_xX$.

\begin{definition}
For any $n,d\in\NN$, the \define{Chow variety} $\Split_d(\PP^n)$ is the variety in $\PP^N$ consisting of all completely reducible degree $d$ forms in $n+1$ variables, i.e.,
\begin{equation*}
\Split_d(\PP^n)=\{[\ell_1\cdots\ell_d]\in\PP^N:\ell_1,\ldots,\ell_d\in R_1\}.
\end{equation*}
\end{definition}

The Chow variety is also known as the \define{split variety} or \define{variety of completely reducible/decomposable forms} and is also denoted in the literature by $\Ch_d(R_1)$ or $\mathbb X_{n,\lambda}$, where $\lambda=(1,\ldots,1)$ is a $d$-tuple.

For an overview of Chow varieties, see \cite[Section 8.6]{Landsberg}.  Note that $\dim\Split_d(\PP^n)=dn$.

\begin{definition}
For a given variety $X\subset\PP^N$, the $s$th \define{secant variety} of $X$ is Zariski closure of the union of all secant $(s-1)$-planes to $X$, i.e.,
\begin{equation*}
\sigma_s(X)=\overline{\bigcup_{x_1,\ldots,x_s\in X}\langle x_1,\ldots,x_s\rangle}.
\end{equation*}

A basic dimension count shows that
\begin{equation*}
\dim\sigma_s(X)\leq\min\{s(\dim X+1)-1,N\}.
\end{equation*}

The right hand side of the above inequality is called the \define{expected dimension} of $\sigma_s(X)$, denoted $\expdim\sigma_s(X)$.  If $\dim\sigma_s(X)<\expdim\sigma_s(X)$, then $\sigma_s(X)$ is \define{defective}.  Otherwise, $\sigma_s(X)$ is \define{nondefective}.
\end{definition}

A famous lemma of Terracini\cite{Terracini} states that if $x_1,\ldots,x_s$ are generic points of a variety $X$ and $y$ is a generic point in $\langle x_1,\ldots,x_s\rangle$, then
\begin{equation*}
T_y\sigma_s(X)=\langle T_{x_1}X,\ldots,T_{x_s}X\rangle.
\end{equation*}

By the product rule from calculus, we see that
\begin{equation*}
\affinecone T_{[\ell_1\cdots\ell_d]}\Split_d(\PP^n)=\sum_{j=1}^d\ell_1\cdots\ell_{j-1}\ell_{j+1}\cdots\ell_d R_1.
\end{equation*}

Combining these results, we obtain the following useful lemma.

\begin{lemma}[Terracini's lemma for Chow varieties]\label{terracini}
If $\ell_{i,j}\in R_1$, $i\in\{1,\ldots,s\}$, $j\in\{1,\ldots,d\}$, are generic, then
\begin{equation*}
\dim\sigma_s(\Split_d(\PP^n))=\dim\sum_{i=1}^s\sum_{j=1}^d\ell_{i,j}\cdots\ell_{i,j-1}\ell_{i,j+1}\cdots\ell_{i,d} R_1-1.
\end{equation*}
\end{lemma}

Note that $\Split_d(\PP^n)$ contains many singular points.  Indeed, if $d\leq n$ and $\{\ell_1,\ldots,\ell_d\}$ is a linearly dependent set in $R_1$, then $[\ell_1\cdots\ell_d]$ will be singular.  Therefore, it is important that we pick generic points for these calculations.  Fortunately, since $\kk$ is algebraically closed and therefore infinite, we may certainly find generic $\ell_i$.

We close this section with an elementary proof of Conjecture \ref{conjecture} for quadratics using Lemma \ref{terracini}.

\begin{theorem}\label{DEqualsTwo}
For all $n\in\NN$,
\begin{equation*}
\dim\sigma_s(\Split_2(\PP^n))=\min\left\{s(2n+1)-2s(s-1),\binom{n+2}{2}\right\}-1.
\end{equation*}

In particular, $\sigma_s(\Split_2(\PP^n))$ is defective if and only if $2\leq s\leq\frac{n}{2}$.
\end{theorem}

\begin{proof}
Let $\ell_0,\cdots,\ell_{2s-1}\in R_1$ be generic and let $V=\sum_{i=1}^{2s-1}\ell_iR_1$, so that $\dim V=\dim\sigma_s(\Split_2(\PP^n))=\dim V-1$ by Lemma \ref{terracini}.

If $2s>n$, then $V=R_2$, and so $\dim V=\binom{n+2}{2}$.

Otherwise, choose linear forms $\ell_{2s},\ldots,\ell_n$ such that $\{\ell_0,\ldots,\ell_n\}$ is a basis for $R_1$.  Then $V$ has basis $\{\ell_i\ell_j:i<2s\text{ or }j<2s\}$.  Then
\begin{align*}
\dim V &= \binom{n+2}{2}-\binom{n-2s+2}{2}\\
&= s(2n+1)-2s(s-1).\qedhere
\end{align*}
\end{proof}
\section{Induction}

We adapt a technique of Abo, Ottaviani, and Peterson for studying secant varieties of Segre varieties \cite{AOP} to secant varieties of Chow varieties.

Recall from Lemma \ref{terracini} that in order to find the dimension of $\sigma_s(\Split_d(\PP^n))$, it suffices to calculate the dimension of a vector space.  Our approach is to specialize some of the $\ell_{i,j}$ and note that the space may then be decomposed into a direct sum of smaller spaces.  These smaller spaces are defined by polynomials with fewer variables and/or smaller degree.  If the smaller spaces have the expected dimension, then the larger space will too.  This allows us to use induction to obtain our results.

Pick nonnegative integers $s,t,u,v$, and choose generic $d$-tuples of linear forms $\vec f_i$ for $i\in\{1,\ldots,s\}$ and $\vec f'''_i$ for $i\in\{1,\ldots,v\}$ and generic $(d+1)$-tuples of linear forms $\vec f'_i$ for $i\in\{1,\ldots,t\}$ and $\vec f''_i$ for $i\in\{1,\ldots,u\}$.  For any integer $k>2$ and $J\subsetneq\{1,\ldots,k\}$, we define the map $\pi_J:R_1^{k}\rightarrow R_{k-|J|}$ by $(\ell_1,\ldots,\ell_k)\mapsto\prod_{j\not\in J}\ell_j$.  Note that we will denote $\pi_{\{j\}}$ by $\pi_j$.

We now define the following subspace of $R_d$.
\begin{align*}
A(n,d,s,t,u,v) =&\sum_{i=1}^s\sum_{j=1}^d\pi_j(\vec f_i)R_1+\sum_{i=1}^t\Span\{\pi_1(\vec f'_i)\}\\
&+\sum_{i=1}^u\sum_{j=1}^{d+1}\Span\{\pi_j(\vec f''_i)\}+\sum_{i=1}^v\pi_1(\vec f'''_i)R_1.
\end{align*}

We also define the function $a$ by
\begin{equation*}
a(n,d,s,t,u,v) = s(dn+1)+td +u(d+1)+v(n+1).
\end{equation*}

\begin{definition}
If $a(n,d,s,t,u,v)\leq\binom{n+d}{d}$, then the 6-tuple $(n,d,s,t,u,v)$ is \define{subabundant}.  If $\dim A(n,d,s,t,u,v)=a(n,d,s,t,u,v)$, then we say that the statement $\statement A(n,d,s,t,u,v)$ is \define{true}.  Otherwise, $\statement A(n,d,s,t,u,v)$ is \define{false}.
\end{definition}

Note that, by Lemma \ref{terracini}, if $(n,d,s,0,0,0)$ is subabundant and $\statement A(n,d,s,0,0,0)$ is true, then $\sigma_s(\Split_d(\PP^n)$ is nondefective.

\begin{theorem}\label{SplittingInductionTheorem}
Suppose $n\geq 2$, $d\geq 3$, $s=s'+s''$, $t=t'+t''$, $u=u'+u''$, and $v=v'+v''$.  Let $S=\{(n-1,d,s'',t''+u',u'',s'+v''),(n-1,d-1,s',t'+v'',s''+u',v'),(n-1,d-2,0,v',s',0)\}$.  If $\vec s$ is subabundant and $\statement A(\vec s)$ is true for all $\vec s\in S$, then $(n,d,s,t,u,v)$ is subabundant and  $\statement A(n,d,s,t,u,v)$ is true.  In particular, if $t=u=v=0$, then $\sigma_s(\Split_d(\PP^n))$ is nondefective.
\end{theorem}

\begin{proof}
In the following, we construct a vector space $V$ whose dimension is bounded above by the dimension of $A(n,d,s,t,u,v)$.  It is defined as the sum of four smaller spaces whose dimensions are known based on our assumptions, providing us with a lower bound.  It turns out that these bounds coincide, and thus we obtain our result.

Let $U=\Span\{x_1,\ldots,x_n\}$.
\begin{itemize}
\item Choose generic $\tuple f_i\in\Span\{x_0\}\times U^{d-1}$ if $i\in\{1,\ldots,s'\}$ and $\tuple f_i\in U^{d}$ if $i\in\{s'+1,\ldots,s\}$.
\item Let $V_1=\sum_{i=1}^s\sum_{i=1}^d\pi_i(\vec f_i)R_1$.
\item Choose generic $\tuple f'_i\in R_1\times\Span\{x_0\}\times U^{d-1}$ if $i\in\{1,\ldots,t'\}$ and $\tuple f'_i\in U^{d+1}$ if $i\in\{t'+1,\ldots,t\}$.
\item Let $V_2=\sum_{i=1}^t\Span\{\pi_1(\vec f'_i)\}$.
\item Choose generic $\tuple f''_i\in\Span\{x_0\}\times U^{d}$ if $i\in\{1,\ldots,u'\}$ and $\tuple f''_i\in U^{d}$ if $i\in\{u'+1,\ldots,u\}$.
\item Let $V_3=\allowbreak\sum_{i=1}^u\sum_{i=1}^{d+1}\Span\{\pi_i(\vec f''_i)\}$
\item Choose generic $\tuple f'''_i\in R_1\times\Span\{x_0\}\times U^{d-2}$ if $i\in\{1,\ldots,v'\}$ and $\tuple f'''_i\in U^{d}$ if $i\in\{v'+1,\ldots,v\}$.
\item Let $V_4=\sum_{i=1}^v\pi_1(\vec f'''_i)R_1$.
\end{itemize}

We define $V=\sum_{i=1}^4V_i$.  Note that, by construction, $\dim V\leq\dim A(n,d,s,t,u,v)$.

Recall that the $V_i$ were constructed carefully with some of the defining linear forms equal to $x_0$ and others which did not involve $x_0$.  We take advantage of this construction by manipulating each of the $V_i$ into a direct sum of smaller spaces.

For $V_1$, we have
\begin{align*}
V_1 &= \sum_{i=1}^{s'}\sum_{j=1}^d\pi_j(\tuple f_i)R_1+ \sum_{i=s'+1}^{s}\sum_{j=1}^d\pi_j(\tuple f_i)R_1\\
&= \sum_{i=1}^{s'}\left(\pi_1(\tuple f_i)R_1+x_0\sum_{j=2}^d\pi_{\{1,j\}}(\tuple f_i)R_1\right)\\
&\quad+\sum_{i=s'+1}^{s}\left(\sum_{j=1}^d\pi_j(\tuple f_i)U+x_0\sum_{j=1}^d\Span\{\pi_j(\tuple f_i)\}\right)\\
&=\left(\sum_{i=1}^{s'}\pi_1(\tuple f_i)U+\sum_{i=s'+1}^{s}\sum_{j=1}^d\pi_j(\tuple f_i)U\right)\\
&\quad\oplus x_0\left(\sum_{i=1}^{s'}\sum_{j=2}^d\pi_{\{1,j\}}(\tuple f_i)U+\sum_{i=s'+1}^{s}\sum_{j=1}^d\Span\{\pi_j(\tuple f_i)\}\right)\\
&\quad\oplus x_0^2\left(\sum_{i=1}^{s'}\sum_{j=2}^d\Span\{\pi_{\{1,j\}}(\tuple f_i)\}\right)\\
&\cong A(n-1,d,s'',0,0,s')\oplus A(n-1,d-1,s',0,s'',0)\\
&\quad\oplus A(n-1,d-2,0,0,s',0),\\
\end{align*}

Via similar calculations, we obtain

\begin{align*}
V_2 &\cong A(n-1,d,0,t'',0,0)\oplus A(n-1,d-1,0,t',0,0)\\
V_3 &\cong A(n-1,d,0,u',u'',0)\oplus A(n-1,d-1,0,0,u',0)\\
V_4 &\cong A(n-1,d,0,0,0,v'')\oplus A(n-1,d-1,0,v'',0,v')\\
&\quad\oplus A(n-1,d-2,0,v',0,0).
\end{align*}

Adding these results together, we get
\begin{align*}
V&\cong A(n-1,d,s'',t''+u',u'',s'+v'')\oplus A(n-1,d-1,s',t'+v'',s''+u',v')\\
&\quad\oplus A(n-1,d-2,0,v',s',0),
\end{align*}
and therefore, by assumption,
\begin{align*}
\dim V &= a(n-1,d,s'',t''+u',u'',s'+v'')\\
&\quad +a(n-1,d-1,s',t'+v'',s''+u',v')+a(n-1,d-2,0,v',s',0)\\
&=s''((n-1)d+1)+t''+u'+u''(d+1)+(s'+v'')n\\
&\quad+s'((n-1)(d-1)+1)+t'+v''+(s''+u')d+v'n\\
&\quad+v'+s'(d-1)\\
&=s(dn+1)+t+u(d+1)+v(n+1)\\
&=a(n,d,s,t,u,v)\geq\dim A(n,d,s,t,u,v).
\end{align*}

Consequently, $\dim V=\dim A(n,d,s,t,u,v)$, and the result follows.
\end{proof}

\section{Results}

Note that the dimensions of $\sigma_2(\Split_d(\PP^n))$ are completely known (see, for example, \cite{Shin2}).  Our goal in this section is to prove conjecture \ref{conjecture} for other small $s$.

We will need the following lemma.

\begin{lemma}\label{F true}
If $\statement A(n,d,s,0,0,0)$ is true and subabundant, then $\statement A(n,d-1,0,0,s,0)$ is true and subabundant.
\end{lemma}

\begin{proof}
For each $i\in\{1,\ldots,s\}$, choose generic $\tuple f_i\in R_1^d$, and let $V_i=\sum_{j=1}^d\pi_j(\tuple f_i)R_1$.

Note that, for each $i$, $\dim V_i\leq d(n+1)$.  But $\prod\tuple f_i\in\pi_j(\tuple f_i)R_1$ for each $m$, and so we can improve this to $\dim V_i\leq d(n+1)-(d+1)=dn+1$.

However, by assumption, we have $\dim\sum_{i=1}^sV_i=s(dn+1)$, and so we must have $\dim V_i=dn+1$.  Furthermore, $V_i\cap V_j=\{0\}$ for $j\neq i$.

Suppose there exist $a_{i,j}\in\kk$ such that $\sum_{i=1}^s\sum_{j=1}^da_{i,j}\pi_j(\tuple f_i)=0$.  Therefore, for nonzero $g\in R_1$, we have $\sum_{i=1}^s\sum_{j=1}^da_{i,j}\pi_j(\tuple f_i)g=0$.  However, $\sum_{j=1}^da_{i,j}\pi_j(\tuple f_i)g\in V_i$ for each $i$, and this implies that we must have $\sum_{j=1}^da_{i,j}\pi_j(\tuple f_i)g=0$.  Consequently, $\sum_{j=1}^da_{i,j}\pi_j(\tuple f_i)=0$, for each $i$.

Suppose $a_{i,j}\neq 0$ for some $i,j$.  Without loss of generality, suppose $j=1$.  Then we have
\begin{align*}
\pi_1(\tuple f_i)&=-\sum_{j=2}^d\frac{a_{i,j}}{a_{i,1}}\pi_j(\tuple f_i)\\
&=-(\tuple f_i)_1\sum_{j=2}^d\frac{a_{i,j}}{a_{i,1}}\pi_{\{1,j\}}(\tuple f_i).
\end{align*}
Therefore, $(\tuple f_i)_1$ occurs twice in $\tuple f_i$, up to scalar multiplication.  This is a contradiction since $\tuple f_i$ was chosen to be generic.  Therefore, $a_{i,j}=0$ for all $i,j$, and thus $\{\pi_j(\tuple f_i):1\leq i\leq s,\,1\leq j\leq d\}$ is a linearly independent set.  Since this set spans $A(n,d-1,0,0,s,0)$ and $a(n,d-1,0,0,s,0)=sd$, the corresponding statement is true.
\end{proof}

\begin{thmn}[\ref{induction on n}]
If, for some $n_0\in\NN$, $s(dn_0+1)\leq\binom{n_0+d}{d}$ and $\sigma_s(\Split_d(\PP^{n_0}))$ is nondefective, then $\sigma_s(\Split_d(\PP^n))$ is nondefective for all $n\geq n_0$.
\end{thmn}

\begin{proof}
By assumption, $\statement A(n_0,d,s,0,0,0)$ is true and subabundant.  Therefore, by Lemma \ref{F true}, $\statement A(n_0,d-1,0,0,s,0)$ is true and subabundant.  It follows by Theorem \ref{SplittingInductionTheorem} that $\statement A(n_0+1,d,s,0,0,0)$ is true and subabundant.  The result follows by induction on $n$.
\end{proof}

Note that, by construction, $s_1(d)(3d+1)\leq\binom{d+3}{d}$ for all $d$.  This allows us to extend Theorem \ref{AboResult} to include forms with a larger number of variables.

\begin{corollary}
If $s\leq s_1(d)$, then $\sigma_s(\Split_d(\PP^n))$ is nondefective.
\end{corollary}

Theorem \ref{induction on n} allows us to reduce determining the nondefectivity of $\sigma_s(\Split_d(\PP^n))$ for fixed $s$ to finitely many cases.  We outline these cases.

\begin{proposition}\label{SmallSAlgorithm}
Fix an $s\in\NN$.  If $\sigma_s(\Split_d(\PP^n))$ is nondefective for all of the following cases
\begin{enumerate}[(i)]
\item $d=3$ and $\min\{n:s<s_2(n)\}\leq n\leq\min\left\{n:s\leq\frac{1}{3n+1}\binom{n+3}{3}\right\}$,
\item $4\leq d\leq\max\{d:s\geq s_2(d)\}$ and $4\leq n\leq\min\left\{n:s\leq\frac{1}{dn+1}\binom{n+d}{d}\right\}$, and
\item $\min\{d:s<s_2(d)\}\leq d\leq\max\{d:s>s_1(d)\}$ and \\$3\leq n\leq\min\left\{n:s\leq\frac{1}{dn+1}\binom{n+d}{d}\right\}$,
\end{enumerate}
then $\sigma_s(\Split_d(\PP^n))$ is nondefective for all $n,d\in\NN$ unless $d=2$ and $n\geq 2s$.
\end{proposition}

\begin{proof}
By the previously known results (in particular, Theorem \ref{AboResult}) or by assumption, we know that, for each $d\geq 3$, $\sigma_s(\Split_d(\PP^n))$ is nondefective if $n\leq\min\left\{n:s\leq\frac{1}{dn+1}\binom{n+d}{d}\right\}$.  It follows for greater $n$ by Theorem \ref{induction on n}.
\end{proof}

Using Proposition \ref{SmallSAlgorithm}, we can check the veracity of Conjecture \ref{conjecture} for any fixed $s$.  Recall that using Lemma \ref{terracini}, we may check the dimension of $\sigma_s(\Split_d(\PP^n))$ by calculating the dimension of a vector space, or equivalently, the rank of a matrix.  This is the usual technique used in the literature.  The author implemented this check in Macaulay2 \cite{M2} for as many values of $s$ as possible using the available computing power.  The code may be found in the ancillary files to the arXiv version of this paper.\\

\begin{thmn}[\ref{SUpperBound}]
If $s\leq 35$, then $\sigma_s(\Split_d(\PP^n))$ is nondefective for all $n,d\in\NN$ unless $d=2$ and $2\leq s\leq\frac{n}{2}$.
\end{thmn}

Note that the $s=35$ calculation took 90738.4 seconds (slightly longer than 1 day).  After this, the machine, which had 32 gigabytes of RAM, ran out of memory and the computations were aborted.

Note that while an increase in computer memory could provide a marginal increase in this upper bound, we should search elsewhere for a substantial improvement to this result.  Note that we decomposed the vector space from Lemma \ref{terracini} into a direct sum of only three smaller vector spaces.  In particular, we were not able to use this technique to consider the \textit{superabundant} case, where we expect the secant variety to fill the ambient space.  Possible future improvements may occur by using a finer decomposition.

\bibliography{ref}{}
\bibliographystyle{habbrv}
\end{document}